\documentclass[a4paper,12pt]{article}

\usepackage{latexsym}
\usepackage{amsmath}
\usepackage{amssymb}
\usepackage{enumerate}

\usepackage{color}

\usepackage{theorem}
\newtheorem{theorem}{Theorem}[section]

\newtheorem{lemma}[theorem]{Lemma}
\newtheorem{corollary}[theorem]{Corollary}

\newtheorem{MT}{Main Theorem}

\theorembodyfont{\rmfamily}
\newtheorem{proof}{\textmd{\textit{Proof.}}}

\makeatletter

\@addtoreset{equation}{section}
\makeatother

\newcommand{\qedd}{\hfill \Box}

\newcommand{\wt}{\widetilde}

\def\diam{\mathop{\mathrm{Diam}}\nolimits}

\title{A generalized maximal diameter sphere theorem
\footnote{Mathematics Subject Classification (2010)\,: 53C22.}
}
\author{Nathaphon BOONNAM}
\date{}
\begin{document}

\maketitle

\begin{abstract}
We prove that if a complete connected $n$-dimensional Riemannian manifold $M$ has radial sectional curvature at a base point $p\in M$ bounded from below by the radial curvature function of a two-sphere of revolution $\wt M$ belonging  to a certain class, then the diameter of $M$ does not exceed that of $\wt M.$ Moreover, we prove that if the diameter of $M$ equals that of  $\wt M,$
then $M$ is isometric to the $n$-model of $\wt M.$
The class of a two-sphere of revolution employed in our main theorem is very wide. For example, this class contains both ellipsoids of prolate type and spheres of constant sectional curvature.
Thus our  theorem contains both the maximal diameter sphere theorem proved by Toponogov \cite{T}   and the radial curvature version by the present author \cite{B} as a corollary.
\end{abstract}

\section{Introduction}
The following result by Bonnet \cite{Bo} is
one of classical and fundamental results in global differential geometry:
``{\it  If the sectional curvature of a complete Riemannian manifold $M$ is not less than a positive constant $H,$ then the diameter of $M$ does not exceed $\pi/\sqrt H. $}"

Toponogov considered the case where the diameter of $M$ attains  the maximal number $\pi/\sqrt H,$ and 
in 1959, he  proved the following theorem called the maximal diameter sphere theorem:
\begin{theorem}
{\normalfont(\cite{T})}\label{th1.1}
Let $M$ be a complete connected $n$-dimensional  Riemannian manifold whose sectional curvature is bounded from below by a positive constant $H$. If the diameter of $M$ equals $\pi/\sqrt H$, then $M$ is isometric to the $n$-dimensional sphere with radius $1/\sqrt H.$
\end{theorem}
Toponogov proved the above theorem by making use of a comparison theorem, which is
now called the Toponogov comparison theorem.
In fact, the Toponogov comparison theorem plays  an important role in the proof of the theorem.
 
Recently, the present author generalized the theorem above by considering  a certain  two-sphere of revolution as a model surface.
\begin{theorem}
{\normalfont(\cite{B})}\label{th1.2}
Let $M$ be a complete connected  $n$-dimensional Riemannian manifold with a base point $p\in M$ whose radial sectional curvature at $p$ is bounded from below by the radial curvature function of a model surface  $\wt M$. Then, the diameter of $M$ does not exceed the diameter of $\wt M$. Furthermore if the diameter of $M$ equals that of $\wt M,$ then $M$ is isometric to the $n$-dimensional model $\wt M^n$.
\end{theorem}
A {\it model surface } $\wt M$ in the above theorem means a two-sphere of revolution $\wt M$ with a pair of poles $\tilde p$ and $\tilde q$ satisfying  the following two properties:
\begin{description}
	\item[{\normalfont (1.1)}] $\wt M$ has a reflective symmetry with respect to the {\it equator} $\{\tilde x\in\wt M| d(\tilde p,\tilde x)=d(\tilde x,\tilde q)\},$ where $d(\cdot,\cdot)$ denotes the Riemannian distance.
	\item[{\normalfont (1.2)}] The Gaussian curvature $G$ of $\wt M$ is strictly decreasing along a meridian from the point $\tilde p$ to the point on the equator.
\end{description}

 Theorem \ref{th1.2} does not contain Theorem \ref{th1.1} as a corollary, since the Gaussian curvature of the unit sphere is constant. 
In this paper, we generalize the maximal diameter sphere theorem as follows. The following theorem contains both theorems above  as a corollary.
\begin{MT}
Let $M$ be a complete connected  $n$-dimensional Riemannian manifold with a base point $p$ such that $M$ has  radial sectional curvature at $p$ bounded from below by the radial curvature function of a two-sphere of revolution $\wt M$ with a pair of poles $\tilde p,\tilde q.$  Suppose that the cut locus of any point on $\wt M$ distinct from both poles is a subset of the half meridian opposite to the point.
Then the diameter of $ M$ does not exceed that of $\wt M$.
Moreover if the diameter of $M$ equals  that of $\wt M,$ then $M$ is isometric to the $n$-model $\wt M^n$ of $\wt M.$
\end{MT}
\begin{corollary}
If a complete connected $n$-dimensional Riemannian manifold $M$ has  radial sectional curvature at  a point $p\in M$ bounded from below by 1, then the diameter of $M$ does not exceed $\pi.$ Moreover if the diameter of $M$ equals $\pi,$ then $M$ is isometric to the $n$-dimensional unit sphere. 
\end{corollary}

If the base point $p$ admits a point $q$ satisfying $d(p,q)=\diam{(\wt M)}$, then it is not difficult to prove that $M$ is isometric to the $n$-dimensional model.
In fact,  under the assumption, Itokawa, Machigashira, and Shiohama proved the maximal diameter sphere theorem  (Theorem 1.4 in \cite{IMS})  for any two-spheres of revolution.

A noncompact version corresponding to the Main Theorem is proved by Shiohama and Tanaka \cite{ShT}.
\begin{theorem}\label{non-compact version}
Let $M$ be a complete noncompact connected Riemannian $n$-manifold with a base point $p\in M,$
at which the radial sectional curvature is bounded from below  by the radial curvature  function of a noncompact complete model surface of revolution $(\widetilde M, dr^2+f(r)^2d\theta^2)$. 
If $\int_1^\infty f(t)^{-2}dt=\infty,$ then $M$ is isometric to the $n$-model $\wt M^n$ 
of $\wt M.$
\end{theorem}
Finally, the author would like to express his deepest thanks to Professor Minoru Tanaka for his various discussions on the maximal diameter sphere theorem. This work was supported by the government budget revenue of Prince of Songkla University. 

\section{Preliminaries and fundamental definitions}

Let us review the definition of  a two-sphere of revolution $(\widetilde M,\tilde g),$ which is defined in \cite{ST}. A Riemannian manifold $(\widetilde M,\tilde g)$   homeomorphic to a 2-dimensional sphere is called a {\it two-sphere of revolution} if  $(\widetilde M,\tilde g)$  admits a point $\tilde p$ whose 
cut locus  has a single cut point $\tilde q,$ and the Riemannian metric $\tilde g$ is expressed as $dr^2+f(r)^2d\theta^2$ on $\widetilde M\setminus \{\tilde p,\tilde q\},$ by using geodesic polar coordinates $(r,\theta)$ around $\tilde p.$ Here   $f$ is  a smooth  function on $[0,d(\tilde p,\tilde q)]$ satisfying $f(0)=f(d(\tilde p,\tilde q))=0.$ Notice that $f$ is positive  on  $(0,d(\tilde p,\tilde q)),$ since the point $\tilde q$ is a unique cut point of $\tilde p.$ Each point of the pair   $\tilde p$ and $\tilde q$ is called  a {\it pole} of the two-sphere of revolution. The typical example of a two-sphere of revolution is the unit sphere in Euclidean space. Incidentally, its Riemannian metric is expressed as $dr^2+\sin^2rd\theta^2$, i.e., the warping function $f$ is $\sin r.$ Each geodesic passing through the poles $\tilde p$ and $\tilde q$ is called a {\it meridian.} Notice that each meridian is a periodic geodesic of length $2d(\tilde p,\tilde q),$ since any meridian is the set of fixed points of an isometry on $\wt M.$
For each two-sphere of revolution $(\wt M,dr^2+f(r)^2d\theta^2)$,
we may define the $n$-model $(\wt M^n,g^*)$ diffeomorphic to the unit  $n$-sphere with a Riemannian metric 
$$g^*=dr^2+f(r)^2d\Theta^2,$$
where $d\Theta^2$ denotes the Riemannian metric of the  $(n-1)$-dimensional unit sphere.

Let $M$ denote a complete connected Riemannian manifold and a point $p$ in $M,$ which is called a {\it base point.}
A 2-dimensional plane $\sigma$ in the tangent space $T_qM$ of $M$ at  a point $q\neq p$
is called a {\it radial plane}  at $q$ if there exists a unit speed minimal geodesic segment
$\gamma:[0,d(p,q)]\to M$ joining $p$ to $q$ such that $\gamma' (d(p,q)  ) \in\sigma,$ and the sectional curvature $K(\sigma)$ of the plane $\sigma$ is called a {\it radial sectional  curvature}. 
The manifold $M$ is said to {\it have  radial sectional curvature at $p$ bounded from below by the radial curvature function of a two-sphere of revolution} $\wt M,$ if for any radial plane $\sigma$ at any point $q$ distinct from $p,$ the radial sectional curvature $K(\sigma) $ is not less than the Gaussian curvature of $\wt M$ at
$\mu(d(\tilde p,\tilde q)).$ Here $\mu$ denotes a unit speed meridian emanating from $\tilde p.$

From now on, $(\wt M,dr^2+f(r)^2d\theta^2)$ denotes a two-sphere of revolution with a pair of poles $\tilde p$ and $\tilde q.$
By scaling the Rimannian metric, we may assume that $\pi=d(\tilde p,\tilde q).$

\begin{lemma}\label{lemd1}
The perimeter of  any geodesic triangle $\triangle(\tilde p\tilde x\tilde y)$ in  $\wt M$
does not exceed $2\pi.$
\end{lemma}

\begin{proof}
Let $\triangle(\tilde p\tilde x\tilde y)$ denote any geodesic triangle in $\wt M.$
By the triangle inequality,
we get
\begin{equation}\label{eq:lemd1-1}
d(\tilde x,\tilde y)\leq d(\tilde x,\tilde q)+d(\tilde q,\tilde y).
\end{equation}
Since it is trivial that
$d(\tilde x,\tilde q)=\pi-d(\tilde p, \tilde x)$ and $d(\tilde q,\tilde y)=\pi-d(\tilde p, \tilde y)$ hold, we obtain
$d(\tilde x,\tilde y)+d(\tilde p,\tilde x)+d(\tilde p,\tilde y)\leq2\pi,$ by \eqref{eq:lemd1-1}.
$\qedd$
\end{proof}
\begin{lemma}\label{lemd2}
The diameter ${\diam}(\wt M):=\max\{ d(\tilde x,\tilde y)| \tilde x,\tilde y\in \wt M \}$ of $\wt M$ is $\pi=d(\tilde p,\tilde q).$
\end{lemma}
\begin{proof}
Choose a pair of points $\tilde x$ and $\tilde y$ satisfying $d(\tilde x,\tilde y)={\diam} (\wt M).$
If $\tilde x$ or $\tilde y$ equals $\tilde p,$ then it is clear that ${\diam}(\wt M)=\pi,$ since $\tilde q$ is the farthest point from $\tilde p.$
Hence we may assume that $\tilde p\ne \tilde x, \tilde y.$
By applying Lemma \ref{lemd1} to the geodesic triangle $\triangle(\tilde p\tilde x\tilde y)$, we obtain
\begin{equation}\label{eq:lemd2-1}
d(\tilde p,\tilde x)+d(\tilde x,\tilde y)+d(\tilde p,\tilde y)\leq 2\pi.
\end{equation}
By the triangle inequality,
we have
\begin{equation}\label{eq:lemd2-2}
d(\tilde x,\tilde y)\leq d(\tilde p,\tilde x)+d(\tilde p,\tilde y).
\end{equation}
By \eqref{eq:lemd2-1} and \eqref{eq:lemd2-2},
we get $d(\tilde x,\tilde y)\leq \pi.$ Thus, the diameter of $\wt M$ is $\pi.$
$\qedd$
\end{proof}


We need  the following  generalized Toponogov comparison theorem for proving the Main Theorem. 
The proof of the following  theorem is  given in \cite{ST}.
\begin{theorem}\label{th2.2}
Let $M$ be a complete connected  $n$-dimensional Riemannian manifold with a base point $p$ such that $M$ has  radial sectional curvature at $p$ bounded from below by the radial curvature function of a two-sphere of revolution $\wt M$ with a pair of poles $\tilde p$ and $\tilde q.$  Suppose that the cut locus of any point on $\wt M$ distinct from its two poles is a subset of the half meridian opposite to the point. Then for each geodesic triangle $\triangle (pxy)$ in $M$, there exists a geodesic triangle $\triangle (\tilde p \tilde x \tilde y)$ in $\wt M$ 
such that
\begin{equation}\label{eq2.1}
d(p,x)=d(\tilde p, \tilde x),\quad 
d(p,y)=d(\tilde p, \tilde y),\quad 
d(x,y)=d(\tilde x, \tilde y),
\end{equation}
and such that
\begin{equation}\label{eq2.2}
\angle(pxy) \geq \angle(\tilde p \tilde x \tilde y),\quad 
\angle(pyx) \geq \angle(\tilde p \tilde y \tilde x),\quad 
\angle(xpy) \geq \angle(\tilde x \tilde p \tilde y).
\end{equation}
Here,  $d(\cdot,\cdot)$ denotes the Riemannian distance and $\angle (pxy)$ denotes the angle at the vertex $x$ of the geodesic triangle $\triangle (pxy)$.
\end{theorem}
Moreover, we also need the following theorem which will be proved along methods developed by Kondo and Tanaka \cite{KT} in Section 3.
\begin{theorem}\label{th2.7}
If $\angle(xpy)=\angle(\tilde x\tilde p\tilde y)$ holds under the same assumptions as in Theorem \ref{th2.2},  then equality  holds for the other two inequalities in \eqref{eq2.2}.
\end{theorem}

\section{Proof of Theorem \ref{th2.7} }

We need the following generalized first variation formula \eqref{eq:lem2.5-0}
in order to prove Theorem \ref{th2.7}.
This lemma is very closely related  to Lemma 2.1 in \cite{IT} and a certain technique  used in the proof of Lemma 2.1 in \cite{IT}  is also used for proving the following lemma.
\begin{lemma}\label{Nlem}
Let $\mu,\eta$ denote geodesics defined on a common interval $(a,b)$ in a complete Riemannian manifold $(M,g).$ Then, for each $t_0\in(a,b)\setminus\psi^{-1}(0),$ where we set  $\psi(t):=d(\mu(t),\eta(t)),$
$\psi_+'(t_0):=\lim_{t\downarrow t_0}{(\psi(t)-\psi(t_0))}/{(t-t_0)}$ exists and

\begin{equation}\label{eq:lem2.5-0}
\psi_+'(t_0)=-g(\mu'(t_0),\gamma_+'(0)) +g( \eta'(t_0), \gamma_+'(\psi(t_0))),
\end{equation}
where $\gamma_+:[0,\psi(t_0)]\to M$ denotes any limit geodesic segment of the sequence
of (unit speed) minimal geodesic segments joining $\mu(t_i)$ to $\eta(t_i)$ as $t_i$ converges  to $t_0$, where $\{t_i \}_i$ denotes a decreasing sequence.
In particular, if there exists a unique minimal geodesic segment $\gamma$ joining $\mu(t_0)$ to $\eta(t_0)$, then $\psi'(t_0)$ exists and 
$\psi'(t_0)=-g(\mu'(t_0),\gamma'(0))+g(\eta'(t_0),\gamma'(\psi(t_0))).$
\end{lemma}

\begin{proof}
Let $\alpha : [0,\psi(t_0)]\to M$ denote any unit speed minimal geodesic segment
joining $\mu(t_0)$ to $\eta(t_0).$
Let $t\in(t_0,b)$ denote any  number and $q_0$ the midpoint of $\alpha.$ 
By the triangle inequality, we get
$\psi(t)\leq d(\mu(t),q_0)+d(q_0,\eta(t)).$
Therefore, by applying  the first variation formula to both functions 
$d(q_0,\mu(t))$ and $d(q_0,\eta(t))$,
we obtain
\begin{equation}\label{eq:lem2.5-1}
\limsup_{t\downarrow t_0}\frac{\psi(t)-\psi(t_0)}{t-t_0}\leq -g(\mu'(t_0),\alpha'(0))+g( \eta'(t_0),\alpha'(\psi(t_0)) ). 
\end{equation}
Choose a sequence $\{h_i\}_i$ of positive  numbers convergent to zero such that
there exists  a limit minimal geodesic segment $\gamma_+$ joining $\mu(t_0)$ to $\eta(t_0)$ of the sequence of  unit speed minimal geodesic segments $\gamma_i$ joining $\mu(t_0+h_i)$ to $\eta(t_0+h_i)$
as $h_i$ goes to zero and that 
\begin{equation}\label{eq:lem2.5-2}
\lim_{i\to \infty} \frac{\psi(t_0+h_i)-\psi(t_0)}{h_i}=\liminf_{t\downarrow t_0}\frac{\psi(t)-\psi(t_0)}{t-t_0}.
\end{equation}

For each $h_i,$ let $q_i$ denote the midpoint of the minimal geodesic segment $\gamma_i$ joining $\mu(t_0+h_i)$ and $\eta(t_0+h_i).$
Then, by the triangle inequality, we obtain
\begin{equation}\label{eq:lem2.5-3}
\psi(t_0)\leq d(\mu(t_0),q_i)+d(q_i,\eta(t_0))
\end{equation}
and
\begin{equation}\label{eq:lem2.5-4}
\psi(t_0+h_i)-\psi(t_0)\geq d(\mu(t_0+h_i),q_i)-d(\mu(t_0),q_i)+d(\eta(t_0+h_i),q_i)-d(\eta(t_0),q_i).
\end{equation}
By imitating the proof of Lemma 2.1 in \cite{IT}, we have
$$\lim_{i\to\infty}\frac{d(\mu(t_0+h_i),q_i)-d(\mu(t_0),q_i)}{h_i}=-g(\mu'(t_0),\gamma_+'(0))$$
and
$$\lim_{i\to\infty}\frac{d(\eta(t_0+h_i),q_i)-d(\eta(t_0),q_i)}{h_i}=g(\eta'(t_0),\gamma_+'(\psi(t_0))).$$
Hence, by \eqref{eq:lem2.5-2} and  \eqref{eq:lem2.5-4},
we get
\begin{equation}\label{eq:lem2.5-5}
\liminf_{t\downarrow t_0}\frac{\psi(t)-\psi(t_0)}{t-t_0}\geq -g(\mu'(t_0),\gamma_+'(0))+g(\eta'(t_0),\gamma_+'(\psi(t_0)).
\end{equation}
Since the equation \eqref{eq:lem2.5-1} holds for any minimal geodesic segment $\alpha$ joining $\mu(t_0)$ to $\eta(t_0),$ we get 
\eqref{eq:lem2.5-0}, by combining \eqref{eq:lem2.5-1} and \eqref{eq:lem2.5-5}.

By reversing the parameters of $\mu$ and $\eta$ in the equation \eqref{eq:lem2.5-0},
we have
\begin{equation}\label{eq:lem2.5-6}
\psi_-'(t_0):=\lim_{t\uparrow t_0}\frac{\psi(t)-\psi(t_0)}{t-t_0}= -g(\mu'(t_0),\gamma_-'(0))+g(\eta'(t_0),\gamma_-'(\psi(t_0)) ),
\end{equation}
where $\gamma_-$ denotes a limit minimal geodesic segment joining $\mu(t_0)$ to $\eta(t_0)$ of the sequence of  unit speed minimal geodesic segments joining $\mu(t_0+h_i)$ to $\eta(t_0+h_i)$
as a sequence $\{h_i\}$  of negative numbers goes to zero.
If there exists a unique minimal geodesic segment $\gamma$ joining $\mu(t_0)$ to $\eta(t_0)$, then it is clear from \eqref{eq:lem2.5-0} and \eqref{eq:lem2.5-6} that
$\psi_+'(t_0)=\psi_-'(t_0),$ and hence $\psi'(t_0)$ exists.
$\qedd$
\end{proof}

The following  lemma is one of  key lemmas for proving  Theorem \ref{th2.7}.
The proof is  similar to that of  Lemma 7.3.2 in \cite{SST}.

\begin{lemma}\label{KL2.1}
Let $(\wt M,dr^2+f(r)^2d\theta^2)$ denote a two-sphere of revolution with a pair of poles $\tilde p$ and $\tilde q.$
Let  $\tilde x\in r^{-1}(0,d(\tilde p,\tilde q))\cap\theta^{-1}(0)$ be a point on $\widetilde M$. 
If two points $\tilde q_1$ and $\tilde q_2$ lying on a common parallel $r=c\in(0,d(\tilde p, \tilde q) )$
satisfy
$$0\leq\theta(\tilde q_1)<\theta(\tilde q_2)\leq\pi,$$
then, 
$$d(\tilde x,\tilde q_1)<d(\tilde x,\tilde q_2)$$
holds.
\end{lemma}

For each real number $\alpha,$
let $\tilde\mu_\alpha:[0,\infty)\to \wt M$ denote the unit speed geodesic 
emanating from the pole $\tilde p$ with $\theta\circ \tilde\mu_\alpha=\alpha$ on 
$(0,d(\tilde p,\tilde q)).$ 
For each $c\in(0,d(\tilde p,\tilde q)),$
$\tilde\sigma_c :[0,2\pi]\to \wt M$ denotes the parallel $r=c,$ i.e., $\tilde\sigma_c(\theta):=\tilde\mu_\theta(c).$

It follows from Lemma \ref{KL2.1} that
for any $a,c\in(0,d(\tilde p,\tilde q)), $ and $\theta_0\in(0,\pi),$
\begin{equation}\label{eq:2.1N}
\liminf_{\theta\downarrow \theta_0}\frac{d(\tilde \mu_0(a),\tilde\sigma_c(\theta))-d(\tilde\mu_0(a),\tilde \sigma_c(\theta_0))}{\theta-\theta_0}\geq 0
\end{equation}

The next lemma states  a stronger conclusion than the equation \eqref{eq:2.1N}.
The proof is the same as that of Lemma 4.2 in \cite{KT}.

\begin{lemma}\label{lem4.2KT}
For any $a_0,c_0\in(0,d(\tilde p,\tilde q))$ and $\theta_0\in(0,\pi),$
there exist constant numbers
$\epsilon_1\in(0,\pi/2)$ and $\delta>0$ 
such that
\begin{equation}\label{eq:2.2N}
|d(\tilde\mu_0(a),\tilde\sigma_c(\theta_2))-d(\tilde\mu_0(a),\tilde\sigma_c(\theta_1))|\geq
(f(c)\sin\epsilon_1)|\theta_2-\theta_1|
\end{equation}
holds for all $a\in(a_0-\delta,a_0+\delta),c\in(c_0-\delta_,c_0+\delta)$, and $\theta_1,\theta_2\in(\theta_0-\delta,\theta_0+\delta).$
\end{lemma}

From Lemma \ref{KL2.1},
it follows that for any positive numbers 
$a,b,$ and $c$ with 
$(a,b,c)\in T,$ where 
$$T:=\{(a,b,c)\in \mathbb{R}^3| a,b,c>0, |a-c|<b<d(\tilde\mu_0(a),\tilde\mu_\pi(c))\},$$
there exists a geodesic triangle $\triangle(\tilde p\tilde y\tilde z)$ in the two-sphere of revolution 
$\wt M$
such that $d(\tilde p,\tilde y)=a, d(\tilde y,\tilde z)=b,$ and  $d(\tilde z,\tilde p)=c.$
Let $\theta(a,b,c)$ denote the angle $\angle(\tilde y \tilde p\tilde z)$ of the geodesic triangle $\triangle(\tilde p\tilde y\tilde z).$
Notice that the angle $\theta(a,b,c)$ is uniquely determined even if there exist more than one minimal geodesic segment joining $\tilde y$ to $\tilde z.$

Since the proof of Lemma 4.3 in \cite{KT} is still valid for the two-sphere of revolution $\wt M,$ we get the following 
lemma. 

\begin{lemma}\label{lem4.3KT}
The function $\theta(a,b,c)$ is locally Lipschitz on the set $T.$
\end{lemma}

Let $\triangle(pxy)$ be a geodesic triangle in $M,$ and maps $\mathbf x$ and $\mathbf y$ from the interval $[0,1]$ into $M$
denote the edges joining $p$ to $x$ and  joining $ p$ to $y$, respectively, i.e., the minimal geodesic segment 
joining  $p={\mathbf x}(0)={\mathbf y}(0)$ to  $x={\mathbf x}(1)$, and joining $p$ to $y={\mathbf y}(1)$, respectively. Here we assume that $\mathbf x$ and $\mathbf y$ are parametrized proportionally to arc-length and that $\angle(xpy)\in(0,\pi).$
Furthermore we assume that for each $t\in(0,1]$ there exists a unique  geodesic triangle 
$\widetilde\triangle(p{\mathbf x}(t){\mathbf y}(t)):=\triangle(\tilde p\tilde x(t)\tilde y(t))$ in $\widetilde M$
corresponding to $\triangle(p{\mathbf x}(t){\mathbf y}(t))$
such that
\begin{equation}\label{eq:4.12KT}
d(\tilde p, \tilde x(t))=d(p,{\mathbf x}(t)), \quad d(\tilde p,\tilde y(t))=d(p,{\mathbf y}(t)
), \quad d(\tilde x(t),\tilde y(t))=d({\mathbf x}(t),{\mathbf y}(t))
\end{equation}
and that
\begin{equation}\label{eq:4.13KT}
\angle(p{\mathbf x}(t){\mathbf y}(t))\geq\angle(\tilde p\tilde x(t)\tilde y(t)), \quad \angle(p{\mathbf y}(t){\mathbf x}(t))\geq\angle(p\tilde y(t)\tilde x(t)).
\end{equation}

From Lemma \ref{lem4.3KT}, the function
$\tilde\theta(t):=\angle(\tilde x(t)\tilde p\tilde y(t))$
is locally Lipschitz on $(0,1].$
Moreover, we get
\begin{lemma}\label{KL2}
The function $\tilde\theta(t)$ is locally Lipschitz  and non-increasing on $(0,1].$

\end{lemma}

\begin{proof}

Putting $a:=d(p,x), b:=d(x,y)$, and $c:=d(p,y),$ 
we get $at=d(p,{\mathbf x}(t)), ct=d(p,{\mathbf y}(t)),$ and $\tilde\theta(t)=\theta(at,d({\mathbf x}(t),{\mathbf y}(t)),ct).$
Since $d({\mathbf x }(t),{\mathbf y}(t))$ is a Lipschitz function, it is trivial from Lemma \ref{lem4.3KT} that $\tilde\theta(t)$ is locally Lipschitz on $(0,1].$
Next, we prove that $\tilde\theta(t)$ is non-increasing. 
It follows from Lemma 7.29 in \cite{WZ} that for almost all $t\in(0,1],$ $\tilde\theta'(t)$ exists and for any $t_0\in(0,1],$
$\tilde\theta(t_0)-\lim_{t\downarrow 0}\tilde\theta(t)=\int_0^{\theta_0}\tilde\theta'(t)dt$
holds. Hence, it is sufficient to prove that  $\tilde\theta'(t)\leq 0$ for almost all $t.$
Choose any $t_0\in(0,1), $ at which $\tilde\theta$ is differentiable.
Since the function $\phi(t):=d({\mathbf x}(t),{\mathbf y}(t))$ is also Lipschitz,
we may assume that  both functions $\tilde\theta(t)$ and $\phi(t)$ are differentiable at $t_0.$
By \eqref{eq:4.13KT}, 
we get
\begin{equation}\label{eq:4.14KT}
\angle (p{\mathbf x}(t_0){\mathbf y}(t_0))\geq\angle(\tilde p\tilde x(t_0)\tilde y(t_0)),\quad\angle (p{\mathbf y}(t_0){\mathbf x}(t_0))\geq\angle(\tilde p\tilde y(t_0)\tilde x(t_0)).
\end{equation}

Let $\tilde\mu_0,\tilde\eta_0 :[0,\infty)\to\wt M$ denote unit speed geodesics emanating
from $\tilde p$ such that the angle made by $\tilde\mu_0$ and $\tilde\eta_0$ at $\tilde p$ equals $\tilde\theta(t_0).$
Thus, we may assume that
$\tilde x(t)=\tilde\mu_0(at)$  for all $t\in(0,1]$ and $\tilde y(t_0)=\tilde\eta_0(ct_0).$
Then, we define a  Lipschitz function $\tilde\psi:[0,1]\to R$ by
$$\tilde\psi(t):=d(\tilde\mu_0(at),\tilde\eta_0(ct)).$$
Since there exists a unique  minimal geodesic segment joining $\tilde\mu_0(at_0)=\tilde x(t_0)$ to $\tilde\eta_0(ct_0)=\tilde y(t_0),$
it follows from Lemma \ref{Nlem} that
\begin{equation}
\tilde\psi'(t_0)=a\cos\angle(\tilde p\tilde x(t_0)\tilde y(t_0))+c\cos\angle(\tilde p\tilde y(t_0)\tilde x(t_0))
\end{equation}
and
\begin{equation}
\phi'(t_0)=a\cos\angle( p{\mathbf x}(t_0){\mathbf y}(t_0))+c\cos\angle(p{\mathbf y} (t_0){\mathbf x}(t_0))
\end{equation}
and hence by \eqref{eq:4.14KT}, we obtain
\begin{equation}\label{eq:2.9}
\phi'(t_0)\leq \tilde\psi'(t_0).
\end{equation}
Supposing that $\tilde\theta'(t_0)>0,$ we will get a contradiction.
Then, there exists a positive $\delta_0$ such that
$\tilde\theta(t)>\tilde\theta(t_0)$ for all $t\in(t_0,t_0+\delta_0).$
From Lemmas \ref{KL2.1} and \ref{lem4.2KT}, there exist $\epsilon_1>0$ and $\delta_1\in(0,\delta_0]$ such that
\begin{equation}\label{eq:4.24KT}
\phi(t)=d(\tilde\mu_0(at),\tilde y(t))\geq\tilde\psi(t)+(f(ct)\sin\epsilon_1)(\tilde\theta(t)-\tilde\theta(t_0))
\end{equation}
for all $t\in(t_0,t_0+\delta_1).$
Since $\phi(t_0)=\tilde\psi(t_0),$  we get
\begin{equation*}
\tilde\psi'(t_0)\geq\phi'(t_0)\geq\tilde\psi'(t_0)+(f(ct_0)\sin\epsilon_1)\tilde\theta'(t_0),
\end{equation*}
by \eqref{eq:2.9} and \eqref{eq:4.24KT},
and hence $\tilde\theta'(t_0)\leq 0.$ This is a contradiction.
Therefore, $\tilde\theta(t)$ is non-increasing on $(0,1].$
$\qedd$
\end{proof}


\begin{lemma}\label{KL3}
If $\tilde\theta(t)$ is constant on $(0,1],$ then 
for  all $t\in(0,1),$
\begin{equation*}
\angle(p{\mathbf x}(t){\mathbf y}(t))=\angle(\tilde p\tilde x(t)\tilde y(t))\quad
\text{and}\quad
\angle(p{\mathbf y}(t){\mathbf x}(t))=\angle(\tilde p\tilde y(t)\tilde x(t))
\end{equation*}
hold.
\end{lemma}
\begin{proof}
Suppose that $\tilde\theta(t)$ is constant on $(0,1].$
Let $\tilde\mu,\tilde\eta:[0,\infty)\to \wt M$ denote unit speed geodesics emanating from $\tilde p$ such that the angle made by $\tilde\mu$ and $\tilde\eta$ at $\tilde p$ equals $\angle(xpy).$
We may assume that for each $t\in(0,1]$, we get $\tilde x(t)=\tilde\mu(at)$ and $\tilde y(t)=\tilde\eta(ct),$ since $\tilde\theta(t)=\lim_{t\downarrow 0}\tilde\theta(t)=\angle(xpy).$
Hence, $d({\mathbf x}(t),{\mathbf y}(t))=d(\tilde\mu(at),\tilde\eta(ct))$
for all $t\in(0,1].$
Since there exists a unique minimal geodesic segment joining $\tilde x(t)$ to $\tilde y(t),$
it follows from Lemma \ref{Nlem} that
$d(\tilde x(t),\tilde y(t))=d(\tilde\mu(at),\tilde\eta(ct))$ is differentiable on $(0,1),$
and that 
\begin{equation}
\frac{d}{dt}d(\tilde x(t),\tilde y(t))=a\cos\angle(\tilde p\tilde x(t)\tilde y(t))+c\cos\angle(\tilde p\tilde y(t)\tilde x(t))
\end{equation}
and 
\begin{equation}
\frac{d}{dt}d({\mathbf x}(t),{\mathbf y}(t))=a\cos\angle(p{\mathbf x}(t){\mathbf y}(t))+c\cos\angle(p{\mathbf y}(t){\mathbf x}(t))
\end{equation}
hold on $(0,1).$
Therefore, by \eqref{eq:4.13KT},
$
\angle(p{\mathbf x}(t){\mathbf y}(t))=\angle(\tilde p\tilde x(t)\tilde y(t)), $
and  $\angle(p{\mathbf y}(t){\mathbf x}(t))=\angle(p\tilde y(t)\tilde x(t))$
hold on  $(0,1).$
$\qedd$
\end{proof}

\noindent
{\it Proof of Theorem \ref{th2.7} }\\ 
Since $\lim_{t\downarrow 0}\tilde\theta(t)=\angle(xpy)=\angle(\tilde x\tilde p\tilde y)=\tilde\theta(1),$ 
$\tilde\theta(t)$ is constant by Lemma \ref{KL2}. Hence, by Lemma \ref{KL3},
we get the conclusion.



\section{Proof of Main Theorem}
In this section we add an additional assumption to the two-sphere of revolution $\wt M $ with a pair of poles $\tilde p,$ and $\tilde q$:
{\it  The cut locus of any point $\tilde x$ on $\wt M$ distinct from both poles is a subset of the half meridian opposite to the point $\tilde x$, i.e., the cut locus of $\tilde x\in\theta^{-1}(0)$ is a subset of $\theta^{-1}(\pi),$ by choosing suitable geodesic polar  coordinates
$(r,\theta).$}

For example, the unit 2-sphere is a trivial example satisfying the property above, and an ellipsoid defined by
$$\frac{x^2+y^2}{a^2}+\frac{z^2}{b^2}=1,$$
where $a<b$ are positive constants, is a non-trivial example of 
a 2-sphere of revolution satisfying the property above (see \cite{ST}.)
   
From now on, $M$ denotes a complete connected $n$-dimensional Riemannian manifold which has radial sectional curvature at  a point $p\in M$ bounded from below by the radial curvature function of the two-sphere of revolution $\wt M.$ 
Furthermore, 
By scaling the Riemannian metrics of $M$ and $\wt M,$ we may assume that $d(\tilde p,\tilde q)=\pi.$



 The following lemma is an easy consequence of Theorem \ref{th2.2}, Lemma \ref{lemd1}, and Lemma  \ref{lemd2}.
\begin{lemma}\label{lem3.0}
The perimeter of any geodesic triangle 
$\triangle(pxy)$
of $M$ does not exceed $2\pi$.
Moreover, 
   the diameter, ${\diam} (M):=\max\{d(x,y)|x,y\in M\},$
of $M$ is at most $\pi$. 
\end{lemma}

\begin{lemma}\label{lem3.1}
   If there exists a pair of points $x,y\in M\setminus\{ p\} $ with 
$d(x,y)=\pi,$ then
the perimeter of  the geodesic triangle $\triangle(pxy)$   equals $2\pi,$ 
and $\angle(xpy)=\pi,$
i.e., the broken geodesic consisting of two minimal geodesic segments joining $x$ to $p$ and joining $p$ to $y$ respectively is a minimal geodesic segment joining $x$ to $y$ which passes through $p.$
\end{lemma}
\begin{proof}
Choose any points $x,y\in M\backslash\{p\}$ satisfying $d(x,y)=\pi$. 
We get a geodesic triangle $\triangle (pxy)$ of $M$. It follows from Lemma \ref{lem3.0} that
\begin{equation}\label{eq3.1}
    d(p,x)+d(p,y)+d(x,y)\leq 2\pi
\end{equation}
holds.
By the triangle inequality, we obtain,
\begin{equation}\label{eq3.2}
    d(p,x)+d(p,y)\geq d(x,y)=\pi.
\end{equation}
Thus, combining \eqref{eq3.1} and \eqref{eq3.2}, we get
\begin{equation}\label{eq3.3}
d(p,x)+d(p,y)=\pi=d(x,y).
\end{equation}
Hence the perimeter of $\triangle(pxy)$ is $2\pi$ and $\angle(xpy)=\pi$. 
$\qedd$
\end{proof}
\begin{lemma}\label{lem.2-sphere}
If there exists a pair of points $\tilde x,\tilde y\in \wt M\setminus\{\tilde p\}$ with
$d(\tilde x,\tilde y)=\pi={\diam}(\wt M),$
then $\tilde y$ is a unique cut point of $\tilde x,$ and for any $\tilde z\in\wt M\setminus\{\tilde x\} $ there exists a minimal geodesic segment joining $\tilde x$ to $\tilde y$ which passes through $\tilde z,$ and hence $d(\tilde x,\tilde z)+d(\tilde z,\tilde y)=d(\tilde x,\tilde y)=\pi.$
\end{lemma}
\begin{proof}
By making use of the same proof of Lemma \ref{lem3.1},
we may prove that  the points $\tilde x$ and $\tilde y$ lie on a common meridian $\mu,$ which is divided in two minimal geodesic segments joining $\tilde x$ to $\tilde y.$
Thus, the cut locus of $\tilde x$ consists of a single point $\tilde y,$ since the cut locus of $\tilde x$ is a subset of the half meridian opposite to the point $\tilde x.$
In particular, the minimal geodesic segment joining $\tilde x$ to any $\tilde z\in\wt M\setminus\{\tilde x\}$ has a geodesic extension which passes through $\tilde y,$ and hence $d(\tilde x,\tilde z)+d(\tilde z,\tilde y)=d(\tilde x,\tilde y).$
$\qedd$
\end{proof}

\begin{lemma}\label{lem3.2}
    Suppose that there exist $x,y\in M\setminus\{p\}$ satisfying $d(x,y)=\pi.$
Then for any point $z\in M\setminus\{p,x,y\},$
there exists a geodesic triangle $\wt\triangle(pxz)$ in the  two-sphere  of revolution corresponding to $\triangle(pxz)$
satisfying $\angle(xpz)=\angle(\tilde x\tilde p\tilde z)$ and \eqref{eq2.1} for $y=z,\tilde y=\tilde z.$
Moreover, there exists  a  minimal geodesic segment joining $x$ to $y$ which passes through $z.$
\end{lemma}
\begin{proof}
By Theorem \ref{th2.2}, there exists a geodesic triangle $\wt\triangle(pxy)$ in the 2-sphere of revolution $\wt M$ corresponding to $\triangle(pxy)$ 
satisfying \eqref{eq2.1}.
Hence, by Lemma \ref{lem.2-sphere},
\begin{equation}\label{eq4.4}
d(\tilde x,\tilde y)=d(x,y)=\pi=d(\tilde p,\tilde x)+d(\tilde p,\tilde y).
\end{equation}
We may assume that 
$\theta(\tilde x)=0, $ and $\theta(\tilde y)=\pi,$ i.e., 
the three points $\tilde x, \tilde p$, and $\tilde y$ lie on 
the  meridian  $\tilde c:=\{\tilde p,\tilde q\}\cup\theta^{-1}\{0,\pi\}$ in this order. 
Choose any $z\in M\setminus\{p,x,y\}.$
By applying Theorem  \ref{th2.2} to the two  geodesic triangles $\triangle(pxz)$ 
and $\triangle(pzy)$, we get  geodesic triangles
$\triangle(\tilde p\tilde x\tilde {z_1})$ and $\triangle (\tilde p\tilde{z_2}\tilde y)$ respectively, 
satisfying 
\begin{equation}\label{eq3.4}  
d(x,z)=d(\tilde x,\tilde z_1), \quad d(z,y)=d(\tilde z_2,\tilde y), \quad d(p,z)=d(\tilde p,\tilde z_1)=d(\tilde p,\tilde z_2)
\end{equation}
and
\begin{equation}\label{eq3.5} 
    \alpha:=\angle(xpz)\geq\angle(\tilde x\tilde p\tilde {z_1})=:\tilde\alpha,
\quad
    \beta:=\angle(zpy)\geq\angle(\tilde{z_2}\tilde p\tilde y)=:\tilde\beta.
\end{equation}
Without loss of generality, we may assume  that the points $\tilde{z_i},i=1,2$ are in the  common hemisphere $\theta^{-1}(0,\pi)$ determined by $\tilde c.$
Since  $d(\tilde x,\tilde y)=\pi$ holds by \eqref{eq4.4}, it follows from Lemma \ref{lem.2-sphere}, from \eqref{eq3.4},
and from the triangle inequality that 
 \begin{equation}\label{eq3.7}
    \pi=d(\tilde x,\tilde{z_1})+d(\tilde{z_1},\tilde y)=d(\tilde x,\tilde{z_2})+d(\tilde{z_2},\tilde y),
\end{equation}
and 
\begin{equation}\label{Neq4.7}
\pi=d(x,y)\leq d(x,z)+d(z,y)=d(\tilde x,\tilde z_1)+d(\tilde z_2,\tilde y). 
\end{equation}
By Lemma \ref{lem3.1}, we get $\alpha+\beta=\angle(xpy)=\pi$.
From  \eqref{eq3.5}, we have $\pi=\alpha+\beta\geq\tilde\alpha+\tilde\beta$.
Thus, $\angle(\tilde x\tilde p\tilde {z_1})=\tilde\alpha\leq\pi-\tilde\beta=\angle(\tilde x\tilde p\tilde {z_2})$, and the two points $\tilde z_1 $ and $\tilde z_2$ are on the parallel
$r=d(p,z),$
by \eqref{eq3.4}.
By applying Lemma \ref{KL2.1}  for the pair of points $\tilde{z_1}$ and $\tilde {z_2}$ on  the parallel $r=d( p,z)$, we obtain
\begin{equation}\label{eq3.6}
    d(\tilde x,\tilde {z_1})\leq d(\tilde x,\tilde {z_2}).
\end{equation}
 By \eqref{eq3.7}, \eqref{Neq4.7} and \eqref{eq3.6}, we have
$$\pi\leq d(\tilde x,\tilde z_1)+d(\tilde z_2,\tilde y)\leq d(\tilde x,\tilde z_2)+d(\tilde z_2,\tilde y)=\pi. $$
This implies that  equality holds in \eqref{eq3.6} and we get, by Lemma \ref{KL2.1}, $\tilde\alpha=\pi-\tilde\beta,$ and 
$\tilde {z_1}=\tilde {z_2}.$
Since $\alpha+\beta=\tilde\alpha+\tilde\beta=\pi,$ 
$\alpha\geq\tilde\alpha,$ and $\beta\geq\tilde\beta$ by \eqref{eq3.5}, 
we obtain 
$\alpha=\tilde\alpha$ and $\beta=\tilde\beta.$
By  \eqref{eq4.4}, and \eqref{eq3.7},
$$d(x,z)+d(z,y)=d(\tilde x,\tilde z)+d(\tilde z,\tilde y)=d(\tilde x,\tilde y)=d( x, y)=\pi$$
holds.
Here we put $\tilde z_1=\tilde z_2=:\tilde z.$
Therefore, we get
$d(x,y)=d(x,z)+d(z,y),$
and it is clear now  that  there exists a minimal geodesic segment joining $x$ to $y$ passing through $z.$
$\qedd$
\end{proof}
{\it Proof of Main Theorem.} 
By scaling the Riemannian metric, we may assume that
$\pi=d(\tilde p,\tilde q)=\diam (\widetilde M).$
It is clear from Lemma \ref{lem3.0} that
the diameter of $M$ does not exceed that of $\widetilde M.$
Suppose that $\diam(M)=\diam(\wt M)=\pi.$
Thus,
there exist 
 $x,y\in M$  satisfying $d(x,y)=\pi.$
If the base point $p$ admits a point $q\in M$ satisfying $d(p,q)=\pi,$ then
it is not difficult to construct an isometry between $M$ and the two-sphere of revolution (see \cite{B}).
Therefore, we may assume that $0<d(p,x)<\pi$ and $0<d(p,y)<\pi.$
Let $\gamma : [0,\infty)\to M$ denote the unit speed geodesic emanating from $p$ passing through
$x$ such that $\gamma|_{[0,d(p,x)]} $ is minimal.
From Lemma \ref{lem3.2}, it follows that $\gamma|_{[ d(p,x),d(p,x)+\pi]}$
is  a minimal geodesic segment joining $x$ to $y.$
In particular, the subarc $\gamma|_{[d(p,x),\pi]}$ is a unique minimal geodesic segment
joining $x$ to $q_\pi:=\gamma(\pi)$.
From Lemma \ref{lem3.2} again, it follows that  the geodesic triangle $\triangle(pxq_\pi)$ admits a geodesic triangle $\wt \triangle(pxq_\pi)$
in the two-sphere of revolution satisfying $\angle(xpq_\pi)=\angle(\tilde x\tilde p\tilde q_\pi)$ and \eqref{eq2.1} for $y=q_\pi,\tilde y=\tilde q_\pi.$ 
Thus, by Theorem \ref{th2.7}, $\pi=\angle(pxq_\pi)=\angle(\tilde p\tilde x\tilde q_\pi),$ $\tilde q_\pi=\tilde q, $ and $d(p,q_\pi)=d(\tilde p,\tilde q_\pi)=\pi$, since
$\pi=d(p,x)+d(x,q_\pi)=d(\tilde p,\tilde x)+d(\tilde x,\tilde q_\pi).$ 
Therefore, the base point $p$ admits a point $q$ satisfying $d(p,q)=\pi.$


\bigskip

\begin{center}
Nathaphon BOONNAM \\

\medskip
Department of Applied Mathematics and Informatics\\
Faculty of Science and Industrial Technology, Prince of Songkla University\\
Muang Suratthani, Suratthani\\
84000 Thailand

\medskip
{\it e--mail } :
{\tt nathaphon.b@psu.ac.th}\\

\end{center}

\end{document}